\documentclass[a4paper,12pt]{article}

\usepackage{fullpage}

\usepackage[latin1]{inputenc}
\usepackage[T1]{fontenc}
\usepackage[english]{babel}

\usepackage{graphicx} 

\usepackage{amsmath,amsthm,amsfonts,amssymb,amsxtra}

\usepackage{array}
\usepackage{enumerate}
\newcommand{\scal}[2]{\langle{#1},{#2}\rangle}
\usepackage{dsfont}
\newcommand{\indicatrice}[1]{\mathds{1}_{{#1}}}
\newcommand{\R}{\mathbb{R}}
\newcommand{\N}{\mathbb{N}}
\newcommand{\sphere}{\mathbb{S}}
\newcommand{\esp}{\mathrm{E}}
\newcommand{\var}{\mathrm{Var}}
\newcommand{\Is}{\mathrm{Is}}
\newcommand{\proba}{\mathbb{P}}

\theoremstyle{plain}                    

\newtheorem{theoreme}{Theorem}
\newtheorem*{theoreme*}{Theorem}

\newtheorem{lemme}[theoreme]{Lemma}
\newtheorem{proposition}[theoreme]{Proposition}

\theoremstyle{definition} 
\newtheorem*{remarque}{Remark}

\title{Isoperimetry for spherically symmetric log-concave probability measures}
\author{Nolwen Huet\textsuperscript{1}}
\date{\today}

\begin{document}
\maketitle
\footnotetext[1]{
Institut de Mathématiques de Toulouse,
UMR CNRS 5219,
Université de Toulouse, 
31062 Toulouse, France.
Email: nolwen.huet@math.univ-toulouse.fr.}
\begin{abstract}
We prove an isoperimetric inequality for probability measures $\mu$ on $\mathbb{R}^n$ with density proportional to $\exp(-\phi(\lambda | x|))$, where $|x|$ is the euclidean norm on $\mathbb{R}^n$ and $\phi$ is a non-decreasing convex function. It applies in particular when $\phi(x)=x^\alpha$ with $\alpha\ge1$. Under mild assumptions on $\phi$,  the inequality is dimension-free if $\lambda$ is chosen such that the covariance of $\mu$ is the identity.
\end{abstract}
\section{Introduction}
In his paper \cite{bobkov-radial-logconcave}, Bobkov studies the spectral gap for spherically symmetric probability measures $\mu$ on $\mathbb{R}^n$ with density
\[
\frac {d\mu(x)}{dx}= \rho(|x|),
\]
where $\rho$ is log-concave.  His main result can be stated as follows.
\begin{theoreme}[Bobkov \cite{bobkov-radial-logconcave}]
\label{thm:bobkov-poinc}
The best constant $P_\mu$ in the Poincaré inequality
\[
\mathrm{Var}_\mu(f) \le P_\mu \int |\nabla f|^2 \,d\mu,\quad \forall f\text{ smooth}
\]
satisfies
\[
\frac{\esp_\mu(|X|^2)}{n}  \le P_\mu \le 12 \frac{\esp_\mu(|X|^2)}{n}  .
\]
In particular, if $\mu$ is isotropic, we get
\[
1\le P_\mu\le 12,
\]
which means a spectral gap not depending on $n$.
\end{theoreme}
Here ``$\mu$ is isotropic'' means that the covariance of $\mu$ is the identity. However, we already know  from the spherically invariance of $\mu$ that the covariance is proportional to the identity. So in our case, the isotropy of $\mu$ reduces merely to $\esp_\mu(|X|^2)=n$.
%\begin{remarque}

If we assume furthermore that $\mu$  itself is log-concave (see \cite{Borell-logconcave} for precisions about log-concave measures), that is to say that $\rho$ is non-decreasing, then one can show an isoperimetric inequality for $\mu$,  thanks to a result of Ledoux \cite{ledoux-geom-bounds} (generalized in \cite{EMilmanConvIsopSpectConc} by E. Milman) bounding the Cheeger constant from below by the spectral gap.

\begin{theoreme}\label{thm:bobkov-cheeger}
There exists a universal constant $c>0$ such that, for any $n\in\N$, all \emph{log-concave} measures $\mu$ on $\R^n$ \emph{spherically symmetric} and \emph{isotropic} satisfy the following isoperimetric inequality:
\begin{equation}\label{eq:cheeger}
\Is_\mu(a)\ge c\ a\wedge (1-a).
\end{equation}
\end{theoreme}

Here $\Is_\mu$ denote the isoperimetric function of $\mu$ and $a\wedge b=\min(a,b)$. We need some notation to define $\Is_\mu$ properly.
Let $A$ be a Borel set in $\R^n$. We define its $\varepsilon$-neighborhood by
\[A_ \varepsilon =\{x\in X; d(x,A)\leq \varepsilon\}.\] The boundary measure of $A$ is
\[
\mu^+(\partial A)=\liminf_{\varepsilon\to0^+}\frac{\mu(A_ \varepsilon)-\mu(A)}{\varepsilon}. 
\]
Now the isoperimetric function of $\mu$ is the largest function $\Is_\mu$ on $[0,1]$ such that for all Borel sets $A$,
\[
\mu^+(\partial A)\geq \Is_\mu\big(\mu(A)\big).
\]
The result of Bobkov answer the KLS-conjecture (\cite{KLS}) in the particular case of spherically symmetric measures. This conjecture asserts that \eqref{eq:cheeger} is true for all log-concave and isotropic measures $\mu$, with a universal constant $c$.

Our aim in this note is to sharpen Theorem \ref{thm:bobkov-cheeger} when $\rho$ is ``better'' than log-concave. For instance,
 the Gaussian measure $\gamma_n$ corresponding  to $\rho(t)=(2\pi)^{-\frac{n}{2}}\exp{-\frac{t^2}2}$, is known to satisfy the log-Sobolev inequality and the following isoperimetric inequality:
\[
\Is_{\gamma_n}(a)\ge c\ \big(
a\wedge (1-a)
\big)
 \sqrt{\log\frac{1}{a\wedge (1-a)}}
\]
with constants not depending on $n$ either.
We can ask what happens for regimes between exponential and Gaussian or even beyond the Gaussian case. This idea has already be developed in \cite{LO,barthe-level,BR-real-sob,BCR-orlicz,BCR-isop} for product measures. 
 
Let $\phi:\R^+\to\R^+$ be a  convex non-decreasing function  of class $\mathcal{C}^2$ such that $\phi(0)=0$. Then we consider the probability measure on $\R^n$
\[
\mu_{n,\phi}(dx) = \frac{e^{-\phi(|x|)} \, dx}{Z_{n,\phi}}
\]
and its associated radial measure on $[0,+\infty)$
\[
\nu_{n,\phi}(dr) =  |\sphere^{n-1}|\frac{r^{n-1} e^{-\phi(r)} \, dr}{Z_{n,\phi}}.
\]
In the particular case $\phi(x)=\phi_\alpha(x)= x^\alpha$ with $\alpha\ge1$, we note $\mu_{n,\alpha}=\mu_{n,\phi_\alpha}$ and $\nu_{n,\alpha}=\nu_{n,\phi_\alpha}$.
We denote by $\sigma_{n-1}$ the uniform probability measure on the unit sphere $\sphere^{n-1}$ of $\mathbb{R}^n$. If $X$ is a random variable of law $\mu_{n,\phi}$, then $|X|$ has the distribution $\nu_{n,\phi}$. Conversely, if $r$ and $\theta$ are independent random variables whose distributions are respectively $\nu_{n,\phi}$  and $\sigma_{n-1}$, then $X=r\theta$ has the distribution $\mu_{n,\phi}$. In view of this representation, we will derive inequalities for $\mu_{n,\phi}$ from inequalities for $\nu_{n,\phi}$ and $\sigma_{n-1}$.

In the subgaussian case, following the results for $\mu_{1,\phi}^{\otimes n}$ from \cite{BCR-isop}, we expect the isoperimetric function of $\mu_{n,\alpha}$ to be equal to a constant depending on $n$, times a symmetric function defined for $a\in[0,\frac12]$ by   
\[L_\alpha(a)=a\left(\log\frac1a\right)^{1-\frac{1}{\alpha}},\]
 and more generally for $\mu_{n,\phi}$, 
 \[L_\phi(a)=\frac{a\log\frac1a}{\phi^{-1}\left(\log\frac1a\right)}.\] 
 
Otherwise, since we are aiming at results which do not depend on $n$ and because of the Central-Limit Theorem (\cite{klartag-TCL}), we cannot expect better isoperimetric profile than the one of the Gaussian measure, proportional to
\[
L_2(a)=a\sqrt{\log\frac{1}{a}}.
\]

The point is to know the exact dependence in $n$ of the constant in front of the term in $a$, and in particular to know whether we recover universal constants in the isotropic case. The main theorems of this paper are stated next.
\begin{theoreme}\label{thm:mualpha}
 There exists a universal constant $C>0$  such that, for every $\alpha\ge 1$, for every $n\in\N^*$, and  every $a\in[0,1]$, it holds
\[
\Is_{\mu_{n,\alpha}}(a)\ge C {n}^{\frac12-\frac1\alpha}\ 
\big(
a\wedge (1-a)
\big)
 \left(\log\frac{1}{a\wedge (1-a)}\right)^{1-\frac{1}{\alpha\wedge2}}.
\]
\end{theoreme}
It can be seen as a corollary of the following more general theorem.
\begin{theoreme}\label{thm:muphi}
Let $\phi:\R^+\to\R^+$ be a  convex non-decreasing function  of class $\mathcal{C}^2$ such that $\phi(0)=0$. If moreover we assume that \begin{enumerate}[i)]
\item\label{thm:H1} $\sqrt{\phi}$ is concave, then for every $n\in\N^*$, and  every $a\in[0,1]$, it holds
\[
\Is_{\mu_{n,\phi}}(a)\ge C  \frac{\sqrt{n}}{\phi^{-1}(n)} \ \phi^{-1}(1)
 \frac{\big(
a\wedge (1-a)
\big)\log\frac1{a\wedge (1-a)}}{\phi^{-1}\left(\log\frac1{a\wedge (1-a)}\right)},
\]
\item\label{thm:H2}  $x\mapsto{\sqrt{\phi(x)}}/{x}$ is
increasing, then for every $n\in\N^*$, and  every $a\in[0,1]$, it holds
\[
\Is_{\mu_{n,\phi}}(a)\ge C\frac{\sqrt{n}}{\phi^{-1}(n)}\ 
\big(
a\wedge (1-a)
\big)
 \sqrt{\log\frac{1}{a\wedge (1-a)}},
\]
\end{enumerate}
where $C>0$ is a universal constant.
\end{theoreme}
Further hypotheses ensure the optimality of these bounds among products of functions of $n$ and functions of $a$, and lead to dimension-free inequalities when normalizing measures to obtain isotropic ones. See Theorem \ref{thm:opt-iso} for more precise statement. 

Note that a straightforward application of Bobkov's inequality for log-concave measures (Theorem \ref{thm:isop-bobkov})  leads to the good profile but with
the wrong dimension dependent constant in front of the isoperimetric
inequality. For instance,  Lemma 4 of \cite{barthe-level} and the computation of exponential moments imply the Theorem \ref{thm:mualpha} with  $n^{-\frac1\alpha}$ instead of $n^{\frac12-\frac1\alpha}$.

\medskip

We introduce in Section \ref{sec:propphi} the different hypotheses made on $\phi$. Then we establish in Section \ref{sec:nu} the isoperimetric inequality for the radial measure. The proof relies on an inequality for log-concave measures due to Bobkov and some estimates of probabilities of balls. Section \ref{sec:tensor} is devoted to the argument of tensorization which yields the isoperimetric inequality  from the ones for the radial measure and the uniform probability measure on the sphere. A cut-off argument is needed to get rid of the case of large radius.  This tensorization relies on a functional version of the inequality whose proof is postponed to Section \ref{sec:func}. We combine the previous results in Section \ref{sec:isop-mu} to prove Theorem \ref{thm:muphi}. Eventually, we discuss the isotropic case and the optimality of the inequalities in Section \ref{sec:opt}.

\section{Hypotheses on $\phi$}\label{sec:propphi}

We make different assumptions on $\phi$, corresponding to the different cases in Theorems \ref{thm:muphi}, \ref{thm:nu}, and \ref{thm:opt-iso}. 
\begin{description}
\item[Hypotheses (H0)] $\phi:\R^+\to\R^+$ is a non-decreasing convex function of class $\mathcal{C}^2$ such that $\phi(0)=0$.
\item[Hypotheses (H1)] $\phi$ satisfies (H0) and $x\mapsto{\sqrt{\phi(x)}}/{x}$ is non-increasing.
\item[Hypotheses (H1')] $\phi$ satisfies (H0) and $\sqrt{\phi}$ is concave.
\item[Hypotheses (H2)] $\phi$ satisfies (H0) and $x\mapsto{\sqrt{\phi(x)}}/{x}$ is non-decreasing.
\item[Hypotheses (H2')] $\phi$ satisfies (H2) and there exists $\alpha\ge2$ such that $x\mapsto\phi(x)/x^\alpha$ is non-increasing.
\end{description}

The next lemma sums up some properties of $\phi$ under our assumptions. 
\begin{lemme}\label{propphi}
\begin{itemize}
\item  
Under (H0), it holds:
\begin{enumerate}[i.]
\item For all $t\ge1$ and $x\ge0$,
\[
 \phi(tx)\ge t\phi(x).
\]
\item For all $t\ge1$ and $y\ge0$,
\[
\phi^{-1}(ty)\le t \phi^{-1}(y).
\]
\item For all $x\ge0$,
\[  x\phi'(x)\ge \phi(x).\]
\end{enumerate}

\item 
Under (H1), it holds:
\begin{enumerate}[i.]
\item For all $t\ge1$ and $x\ge0$,
\[
t\phi(x)\le \phi(tx) \le t^2\phi(x).
\]
\item For all $t\ge1$ and $y\ge0$,
\[
\sqrt{t} \phi^{-1}(y)\le \phi^{-1}(ty)\le{t} \phi^{-1}(y).
\]
\item For all $x\ge0$,
\[{\phi(x)}\le {x}\phi'(x) \le 2{\phi(x)}{}.\]
\item For all $t\ge1$ and $x\ge0$,
\[
 \phi'(tx) \le 2t\phi'(x).
\]
\end{enumerate}

\item 
Under (H2), it holds:
\begin{enumerate}[i.]
\item For all $t\ge1$ and $x\ge0$,
\[
 \phi(tx)\ge t^2\phi(x).
\]
\item For all $t\ge1$ and $y\ge0$,
\[
\phi^{-1}(ty)\le \sqrt{t} \phi^{-1}(y).
\]
\item For all $x\ge0$,
\[  x\phi'(x)\ge 2\phi(x).\]
\end{enumerate}

\end{itemize}
\end{lemme}

\section{Isoperimetry for the radial measure $\nu_{n,\phi}$} \label{sec:nu}
In order to deal with $\mu_{n,\phi}$, a first step is to establish a similar isoperimetric inequality for its radial marginal.
\begin{theoreme}\label{thm:nu}
There exists a universal constant $C>0$ such that for every $n\in\N^*$,  every $a\in[0,\frac12]$, and every function $\phi$,
\begin{enumerate}[i)]
\item if $\phi$ satisfies (H1) then
\[
 \Is_{\nu_{n,\phi}}(a)\ge C\frac{\sqrt{n}}{\phi^{-1}(n)}\  \phi^{-1}(1)  \frac{a\log\frac{1}{a}}{\phi^{-1}\left(\log\frac{1}{a}\right)}.
\]
\item if $\phi$ satisfies (H2) then
\[
\Is_{\nu_{n,\phi}}(a)\ge C \frac{\sqrt{n}}{\phi^{-1}(n)}\  a\sqrt{\log\frac1a}.
\]
\end{enumerate}
\end{theoreme}

As $\nu_{n,\phi}$ is a log-concave measure, we can apply the isoperimetric inequality shown by Bobkov in \cite{bobkov-logconcave}.
\begin{theoreme}[Bobkov \cite{bobkov-logconcave}]\label{thm:isop-bobkov}
If $\mu$ is a log-concave measure on $\mathbb{R}^n$, then for all Borel sets $A$, for all $r>0$, and for all $x_0\in\R^n$,
\begin{equation}%
2r\mu^+(\partial A)\ge \mu(A) \log \frac1{\mu(A)}
 + \mu(A^\complement{}) \log \frac1{\mu(A^\complement{})} +
\log \mu\{|x-x_0|\le r\},
 \label{eq:isop-bobkov}
\end{equation}
where $A^\complement{}$ denotes the complement of $A$.
\end{theoreme}

One chooses $r$ as small as possible but with $\mu\{|x-x_0|\le r\}$ large enough,  such that the sum of the two last terms is non-negative.  This requires explicit estimates of probabilities of balls. In our case, we will use two different estimates valid for two ranges of $r$, leading to inequalities for two ranges of $a$.

The first lemma is due to Klartag \cite{klartag-TCL}. The balls are centered at the maximum of density in order to capture a large fraction of the mass.
\begin{lemme}[Klartag \cite{klartag-TCL}]\label{lem:klartag}
Let $\nu(dr)=r^{n-1}\rho(r)\,dr$ be a probability measure on $\R^+$ with $\rho$ a log-concave function  of class $\mathcal{C}^2$. Let $r_0$ be the point where the density reaches its maximum. Then,  
\[
\forall \delta\in[0,1],\quad
\nu\{|r-r_0|\ge \delta r_0\}\le  C_1  e^{-c_1n\delta^2}
\]
where $C_1>1$ and $0<c_1<1$ are universal constants.
\end{lemme}
Bobkov's inequality combined with the latter lemma leads to the following proposition.
\begin{proposition}\label{prop:nu-big}
There exist two universal constants $c>0$ and $C>0$ such that for all functions $\phi$ satisfying (H0) and all $n$ large enough to ensure $e^{-cn}<\frac12$, it holds
\[
\forall a \in\left[e^{-cn},\frac12\right] ,\quad
\Is_{\nu_{n,\phi}} (a) \ge C \frac{\sqrt{n}}{\phi^{-1}(n)}\ a \sqrt{\log\frac{1}{a}}.
\]
\end{proposition}
\begin{proof}
Let $C_1$ and $c_1$ be the constants given by Lemma \ref{lem:klartag}.
Let $K>0$ and set \[\delta=\sqrt{\frac{K\log\frac1a}{c_1n}}.\]
Choose $a\in\left[\exp\left({-\frac{c_1n}K}\right),\frac12\right]$ and  $K>\frac{\log C_1}{\log2}$. It follows that $\delta\le1$ and $1 - C_1 a^{K}>0$. Then Lemma \ref{lem:klartag} implies
\begin{equation}\label{eq:nu1}
(1-a) \log \frac1{1-a} +
\log \nu_{n,\phi}\{|r-r_0|\le \delta r_0\}\ge 
(1-a) \log \frac1{1-a} + \log (1 - C_1 a^{K}).
\end{equation}
The right-hand term of \eqref{eq:nu1} cancels at 0 and is concave in $a$ on $[0,\frac12]$ if $K\ge1$. Take $K$ large enough such that it is also non-negative at $\frac12$. Thus, by concavity, it is non-negative on $[0,\frac12]$. So Bobkov's formula \eqref{eq:isop-bobkov} yields
\[
\Is_{\nu_{n,\phi}} (a) \ge \frac12
\sqrt{\frac{c_1n}{Kr_0^2}}\ a \sqrt{\log\frac{1}{a}}.
\]
It remains to estimate the point $r_0$ where the density of $\nu_{n,\phi}$ reaches its maximum. The differentiation of the density leads to 
\[
r_0\phi'(r_0)=n-1.
\]
By Lemma \ref{propphi}, $\phi(r_0)\le n-1$. Thus
\[
r_0\le \phi^{-1}(n).
\]
\end{proof}
Let us remark that under (H1),  for all $a\le\frac12$,
\[
\sqrt{\log\frac{1}{a}}\ge \sqrt{\log 2}\ \phi^{-1}(1) \frac{\log\frac{1}{a}}{\phi^{-1}\left(\log\frac{1}{a}\right)}.
\]
So the latter proposition implies a stronger inequality that the one required under (H1), but only for large enough sets.

To cope with smaller sets, we need another estimate for balls with greater radius.
\begin{lemme}\label{lem:queue-phi}
Let $\phi$ be a function satisfying (H0) and $n\in\N^*$. Then for all $r\ge \phi^{-1}(2n)$,
\[
\nu_{n,\phi}\{(r,+\infty)\}\leq F_{n,\phi}(r)=\left( \frac{er}{\phi^{-1}(n)}\right)^n e^{-{\phi(r)}}\le1.
\]
\end{lemme}
Note that this tail bound gives estimates of probability of balls centered at 0 for $\nu_{n, \phi}$, but also for $\mu_{n, \phi}$ since 
\[
\nu_{n,\phi}\{(r,+\infty)\} = \mu_{n,\phi}\{|x|\ge r\}.
\]
This lemma can thereby be used to derive isoperimetric inequalities from Bobkov's formula for both measures.
\begin{proof}
The main tool is integration by part.
\begin{align*}
\int_r^{+\infty} t^{n-1} e^{-\phi(t)}\,dt
&=\int_r^{+\infty} \frac{t^{n-1}}{\phi'(t)}\ \phi'(t)e^{-\phi(t)}\,dt\\
&= \frac{r^{n-1}}{\phi'(r)}e^{-\phi(r)} + \int_r^{+\infty} 
\left[\frac{n-1}{t\phi'(t)}-\frac{\phi''(t)}{\big(\phi'(t)\big)^2}\right]\ {t^{n-1}}e^{-\phi(t)}\,dt\\
&\le \frac{r^{n-1}}{\phi'(r)}e^{-\phi(r)} + \int_r^{+\infty} 
\frac{n-1}{t\phi'(t)}\ {t^{n-1}}e^{-\phi(t)}\,dt.\\
\end{align*}
If $t\ge r\ge\phi^{-1}(2n)\ge \phi^{-1}\big(2(n-1)\big)$, then $t\phi'(t)\ge 2 (n-1)$. So the last integral in the above inequality is less than $\frac12 \int_r^{+\infty} t^{n-1} e^{-\phi(t)}\,dt$. Moreover $r\phi'(r)\ge 2 n$. Hence
\begin{equation*}%\label{IPP}
\int_r^{+\infty} t^{n-1} e^{-\phi(t)}\,dt \le 2 \frac{r^{n-1}}{\phi'(r)}e^{-\phi(r)}\le \frac{r^{n}}{n}e^{-\phi(r)}.
\end{equation*}

It remains to deal with the normalization constant which makes $\nu_{n,\phi}$ a probability measure:
\begin{align*}
\int_0^{+\infty} n t^{n-1} e^{-\phi(t)}\,dt &\ge
\int_0^{\phi^{-1}(n)} n t^{n-1} e^{-\phi(t)}\,dt\\
&\ge e^{-n}\int_0^{\phi^{-1}(n)} n t^{n-1} \,dt = \left(\frac{\phi^{-1}(n)}{e}\right)^n.
\end{align*}

Putting all together, we get the desired bound on the tail of $\nu_{n,\phi}$:
\[
\nu_{n,\phi}\{(r,+\infty)\}=\frac{\int_r^{+\infty} t^{n-1} e^{-\phi(t)}\,dt}{\int_0^{+\infty}  t^{n-1} e^{-\phi(t)}\,dt}
\le \left( \frac{er}{\phi^{-1}(n)}\right)^n e^{-{\phi(r)}}.
\] 
Then one can show that the bound is non-increasing for $r\ge\phi^{-1}(n)$ and is equal to 1 for $r=\phi^{-1}(n)$.
\end{proof}
Then, we show  an isoperimetric inequality simultaneously for $\mu_{n,\phi}$ and $\nu_{n,\phi}$ in the range of small sets. 
\begin{proposition}
\label{prop:nu-small-h0}
For every $c>0$, there exists $C>0$  such that for all functions $\phi$ satisfying the hypotheses of (H0),
\[
 \forall a\in\left[0, 
e^
{-c n}\wedge\frac12
\right],\quad
 \Is_{\mu}(a)\ge C \ \frac{a\log\frac1a}{\phi^{-1}\left(\log\frac1a\right)},
\]
where $\mu$ stands for $\mu_{n,\phi}$ or $\nu_{n,\phi}$.
\end{proposition}
Note that this is worth showing the result for every $c>0$. Indeed, to prove Theorem \ref{thm:nu}, we combine this result with Proposition \ref{prop:nu-big} where this constant is already fixed but unknown.
\begin{proof}
As before, we start from \eqref{eq:isop-bobkov} and set 
$r(a)=\phi^{-1}\left(K\log\frac1a\right)$, where $K$ is a constant large enough to ensure
\begin{equation}\label{Kphi2}
Kc\ge 2,
\end{equation}
\begin{equation}\label{Kphi1}
K-1\ge \frac1c,
\end{equation}
\begin{equation}\label{Kphi3}
eKc \exp\left(-{(K-1)c}\right)\le \frac12.
\end{equation}
 By Lemma \ref{propphi}, $r\le K\phi^{-1}\left(\log\frac1a\right)$, as $K>1$. So the result is deduced from Bobkov's inequality \eqref{eq:isop-bobkov} provided that
\begin{equation}\label{eq:h0-pos}
(1-a)\log\frac1{1-a} + \log \mu_{n,\alpha}\{|x|\le r\} 
\ge 0.
\end{equation}
Now, by concavity,
\[
\forall x\in\left[0,\frac12\right], \quad (1-x)\log\frac1{1-x}\ge \log2\ x,
\quad \text{and} \quad \log(1-x) \ge - 2\log2\ x.
\]
So, for all $a\in[0,\frac12]$,  
\begin{equation*}
(1-a)\log\frac1{1-a} + \log \mu_{n,\alpha}\{|x|\le r\} 
\ge \log2\Big( a - 2 F_{n,\phi}(r)\Big)
\ge 0,
\end{equation*}
as soon as 
\[
r\ge\phi^{-1}(2n)
\quad \text{ and } \quad
 F_{n,\phi}(r)\le \frac{a}{2}.
 \]

Assume that $a\le \exp({-c n})\wedge\frac12$. Then $r\ge\phi^{-1}(Kcn)\ge \phi^{-1}(2n)$ by \eqref{Kphi2}.   Let us define the function $G$  by
\[
G(a)=\frac{ F_{n,\phi}\big(r(a)\big)}{a}.
\]
Then \eqref{eq:h0-pos} holds  as soon as $G(a)\le \frac12$.
To handle this, it is easier to look on $G$ as a function of $r$. We know that $a=\exp\left(-\frac{\phi(r)}{K}\right)$. So
\[
G(a)= \left(\frac{er}{\phi^{-1}(n)}\right)^n\exp\left(-{\phi(r)}\Big(1-\frac1K\Big)\right).
\]
This function is non-increasing in $r$ when 
\[r\phi'(r)\ge \frac{n}{1-\frac1K}.\]
This is the case if $r\ge \phi^{-1}\left(\frac{Kn}{K-1}\right)$. Moreover $\phi^{-1}(Kcn)\ge\phi^{-1}\left(\frac{Kn}{K-1}\right)$ by \eqref{Kphi1}. Thus,  when $a\le \exp({-c n})$,
\[
G(a)\le G\big(\exp({-c n})\big) \le \Big[eKc \exp\big(-{(K-1)c}\big)\Big]^n\le \frac1{2^n}\le\frac12.
\]
\end{proof}
Under (H1), this result is again stronger than the one required since then
\[
1\ge \frac{\sqrt{n}}{\phi^{-1}(n)}\phi^{-1}(1).
\]
We could also derive the required inequality under (H2), but with $\sqrt{n}/\phi^{-1}(cn)$ instead of $\sqrt{n}/\phi^{-1}(n)$. So we prefer to prove it directly, following the above proof.

\begin{proposition}
\label{prop:nu-small-h2}
For every $c>0$, there exists $C>0$  such that for all functions $\phi$ satisfying (H2),
\[
 \forall a\in\left[0, 
e^
{-c n}\wedge\frac12
\right],\quad
 \Is_{\mu}(a)\ge C\frac{\sqrt{n}}{\phi^{-1}(n)} \ a\sqrt{\log\frac1a},
\]
where $\mu$ stands for $\mu_{n,\phi}$ or $\nu_{n,\phi}$.
\end{proposition}
\begin{proof}
We set 
\begin{equation}\label{eq:r=a}
r(a)=\sqrt{\frac{K\big(\phi^{-1}(n)\big)^2}{n}\log\frac1a},
\end{equation}
where $K$ is a constant large enough to verify
\begin{equation*}\label{Kphi22}
Kc\ge 2,
\end{equation*}
\begin{equation*}\label{Kphi12}
K-1\ge \frac1{2c},
\end{equation*}
\begin{equation*}\label{Kphi32}
e\sqrt{Kc} \exp\left(-{(K-1)c}\right)\le \frac12.
\end{equation*}
Assume that $a\le \exp({-c n})\wedge\frac12$, then 
\[r\ge\sqrt{Kc}\ \phi^{-1}(n)\ge \sqrt{\frac{Kc}{2}} \phi^{-1}(2n)\ge\phi^{-1}(2n).\] 
So we can use the estimate from Lemma \ref{lem:queue-phi}. Consider as before
\[
G(a)=\frac{ F_{n,\phi}\big(r(a)\big)}{a}.
\]
Then, as explained in the proof of Proposition \ref{prop:nu-small-h0},  Bobkov's formula \eqref{eq:isop-bobkov} yields the required isoperimetric inequality  as soon as 
\[G(a)\le \frac12.\]
From \eqref{eq:r=a}, we deduce 
\[a=\exp\left(-\frac{nr^2}{K\big(\phi^{-1}(n)\big)^2}\right).\]
 So if we express $G$ as a function of $r$,
\[
G(a)= \left(\frac{er}{\phi^{-1}(n)}\right)^n\exp\left(-{\phi(r)}+\frac{nr^2}{K\big(\phi^{-1}(n)\big)^2}\right).
\]
The derivative $\partial_r G^{\frac1n}$ is of the same sign as
\[
1+\frac{2r^2}{K\big(\phi^{-1}(n)\big)^2}-\frac{r\phi'(r)}n.
\]
Under hypothesis (H2), $r\phi'(r)\ge2\phi(r)\ge2 n \left({r}/{\phi^{-1}(n)}\right)^2$ as soon as $r\ge\phi^{-1}(n)$. Thus, when $r\ge \sqrt{Kc}\ \phi^{-1}(n)$,
\[
1+\frac{2r^2}{K\big(\phi^{-1}(n)\big)^2}-\frac{r\phi'(r)}n
\le
1+\frac{2r^2}{\big(\phi^{-1}(n)\big)^2}\left(\frac1K -1\right)
\le 1+2Kc\left(\frac1K -1\right)\le0,
\]
since $\frac1K -1<0$.
So $G$ is non-increasing in $r$
 when $r\ge \sqrt{Kc}\ \phi^{-1}(n)$, and  for all $a\le \exp({-c n})$, it holds
\[
G(a)\le \left(e\sqrt{Kc}\right)^n \exp\bigg(cn-{\phi\Big(\sqrt{Kc}\ \phi^{-1}(n)\Big)}\bigg)
 \le \left[e\sqrt{Kc} \exp\left(-{(K-1)c}\right)\right]^n\le \frac12.
\]
We have again used Hypothesis (H2) which ensures $\phi\Big(\sqrt{Kc}\ \phi^{-1}(n)\Big)\ge Kcn$.
\end{proof}

Combining  Proposition \ref{prop:nu-big} for big sets, and Proposition \ref{prop:nu-small-h0} or Proposition \ref{prop:nu-small-h2} for small sets yields Theorem \ref{thm:nu}.

\section{Tensorization and cut-off argument}
\label{sec:tensor}
%\subsection{Preliminaries}
We derive the isoperimetric inequality for $\mu_{n,\alpha}$ by tensorization from the ones for the radial measure and the uniform probability measure on the sphere, following the idea of the proof by Bobkov of Theorem \ref{thm:bobkov-poinc}. For that purpose, we need a functional version of our isoperimetric inequality. In \cite{bobkov-isop-uniform} and  \cite{Barthe-spher}, the authors give conditions so that isoperimetric inequalities translate into functional inequalities. Actually this works in our setting as explained in Section \ref{sec:func}.

Let $\kappa>0$. Let $J:[0,1]\to\R^+$ be a continuous convex function symmetric with respect to $1/2$, with J(0)=J(1)=0, and such that the following property holds : for any measure $\mu$ on $\R^d$ and constant $C\ge0$, if
\[
\Is_\mu\ge C J,
\]
then for all smooth functions $f:\R^d\to[0,1]$,
\[
\kappa J\left(\int f\,d\mu\right)\leq 
\int J(f)\,d\mu + \frac1{C} \int |\nabla f|\,d\mu.
\]

\begin{remarque}
Ideally, one would expect $\kappa=1$. For instance the latter inequality implies the former one and is tight for constant functions only in the case $\kappa=1$. However this does not matter here as we tensorize only once.
\end{remarque}
For such profiles $J$, we can show the following proposition.
\begin{proposition}\label{thm:tensor}
Let $\mu$ be a measure on $\R^n$ with radial measure $\nu$. Assume that there exists positive constants $C_\nu$ and $C_{\sigma_{n-1}}$ such that
\[ \Is_\nu \ge C_\nu J \quad\text{ and }\quad \Is_{\sigma_{n-1}}\ge C_{\sigma_{n-1}}J.
\]
There exist $\kappa_1, \kappa_2>0$ depending only on $\kappa$ such that, for every $n\in\N^*$, for every $r_2>r_1>0$ and $a$ such that
\begin{equation}\label{cond-r}
  r_2-r_1 \ge \frac1{C_{\nu}J({\textstyle \frac12})},
\end{equation}
\begin{equation}\label{cond-a}
\kappa_1\,\nu\{[r_1 ,+\infty)\}\le a \le \frac12,
\end{equation}
it holds
\[
\Is_{\mu}(a)\ge \kappa_2 \min \Big({C_{\nu}},\frac{C_{\sigma_{n-1}}}{r_2}\Big)\ 
J(a).
\]
\end{proposition}
\begin{proof}
Let $f:\R^n\to[0,1]$ be a smooth function. We recall some facts on radial and spherical differentiation. If we define $g$ on $\mathbb{R^+}\times\sphere^{n-1}$ by $g(r,\theta)=f(r\theta)$, then the partial derivatives of $g$ can be computed as follows:
\begin{align*}
\partial_r g& = \scal{\nabla f}{\theta}, \\
\nabla_{\theta} g &= r\, \Pi_{\theta^\perp}(\nabla f),
\end{align*}
where $\Pi_{\theta^\perp}$ is the orthogonal projection on ${\theta^\perp}$. Hence,
\begin{align*}
\nabla f &= \partial_r g \,\theta + \frac1r \nabla_{\theta} g ,\\
|\nabla f|^2 &= |\partial_r g|^2 + \frac1{r^2} |\nabla_{\theta} g|^2.
\end{align*}

First,  we apply the functional inequality for $\sigma_{n-1}$ to
the function $F$ defined on $\sphere^{n-1}$ by
\[
F(\theta)= \int f(r\theta)\,d\nu(r).
\]
As $\int F \,d\sigma_{n-1}= \int f \,d\mu$, this yields
\[
\kappa J\left(\int f \,d\mu \right)
\le \int  J(F)\,d\sigma_{n-1}
+ \frac1{C_{\sigma_{n-1}}} \int |\nabla_{\sphere^{n-1}}F|\,d\sigma_{n-1}.
\]
On one hand,
\[
\nabla_{\sphere^{n-1}}F(\theta)= \int r\, \Pi_{\theta^\perp}(\nabla f)(r\theta)\,d\nu(r).
\]
On the other hand, we can use the inequality for $\nu$ to bound $J(F)$. Indeed, 
for all $\theta\in\sphere^{n-1}$,
\[
\kappa J\big(F(\theta)\big)\leq \int J\big(f(r\theta)\big) \,d\nu(r) + \frac1{C_{\nu}} \int|\partial_r f(r\theta)|\,d\nu(r).
\]
Putting all together, 
\begin{multline}\label{ineqf}
\kappa^2 J\left(\int f \,d\mu \right)
\le \int J(f) \,d\mu\\
+ \frac1{C_{\nu}} \int|\partial_r f|\,d\mu
+ \frac \kappa{C_{\sigma_{n-1}}} \int|x|\,\left|\Pi_{\theta^\perp}(\nabla f)\right|\,d\mu(x).
\end{multline}
%
%\subsection{Cut-off argument}
%

We would like to get $|x|$ out of the last integral. As it is not bounded, we use a cut-off argument similar to the one in Sodin's article \cite{sasha-lp}, while simpler in our case. Heuristically, we use the fact that on ``a set of large measure'', $|x|$ is almost constant, close to its expectation for instance. 
Let us introduce a cut-off function $h(r\theta)=h_1(r)$ with 
\[
h_1=
\left\{
\begin{array}{>{\displaystyle}l@{}l}
1 &\text{ on }[0,r_1) \\
\frac{r_2 - r}{r_2-r_1} &\text{ on }[r_1 ,r_2]\\
0 &\text{ on }(r_2,+\infty)
\end{array}
\right.
\] 
with $0<r_1<r_2$ to be chosen later (typically of the same order as $\esp_{\mu}|X|$).  It holds
\[
\nabla(fh)=h\nabla f + f\nabla h,
\]
thus
\begin{align*}
|\partial_r (fh)|&\le |\partial_r f| + ||f||_\infty |\partial_r h|,
\\
{\text{\LARGE\strut}}|\Pi_{\theta^\perp}\big(\nabla (fh)\big)| &\le h\, |\Pi_{\theta^\perp}(\nabla f)|.
\end{align*}
As $h=0$ if $|x|>r_2$, 
\[
\int|x|\,\left|\Pi_{\theta^\perp}\big(\nabla (fh)\big)\right|\,d\mu(x)\le
r_2\int\left|\Pi_{\theta^\perp}(\nabla f)\right|\,d\mu(x).
\]
Besides, we can bound the derivative of $h$ so that
\[
\int|\partial_r h|\,d\mu\le \frac{\nu\big([r_1,r_2]\big)}{r_2-r_1}.
\]
Finally, Inequality \eqref{ineqf} applied to $fh$ yields
\begin{align}
\lefteqn{
\kappa^2 J\left(\int fh \,d\mu \right)
- \int J(fh) \,d\mu 
-   \frac{||f||_\infty\nu\big([r_1,r_2]\big)}{C_{\nu}(r_2-r_1)}
}
\phantom{\text{un peu plus à droite, encore}}\nonumber
\\
&\le
\max\Big(\frac1{C_{\nu}},\frac{\kappa r_2}{C_{\sigma_{n-1}}}\Big)
\left(
\int |\partial_rf| +\left|\Pi_{\theta^\perp}(\nabla f)\right|\,d\mu
\right)
\nonumber\\
&\le 
\sqrt{2}\max\Big(\frac1{C_{\nu}},\frac{\kappa r_2}{C_{\sigma_{n-1}}}\Big)
\int |\nabla f| \,d\mu.\label{ineq-fh}
\end{align}
Hence we have almost the functional inequality for $f$ and $\mu$ with an additional term that we expect to be negligible. 
It is easier to look at functions approximating characteristic functions to go back from $fh$ to $f$ in the left hand term.

Let $A\subset\R^n$ be a closed set of measure $a\le \frac12$. Let $K>0$ and $t\in(0,1)$ constants to be chosen later.
Assume the following constraints on $r_1$, $r_2$, and $a$:\begin{equation*}
C_{\nu} (r_2-r_1) \ge K,
\end{equation*}
\begin{equation*}
\nu\{[r_1 ,+\infty)\} \le t a.
\end{equation*} 
Then it holds
\begin{align*}
\mu\{\indicatrice{A} h=1\}
&\ge\mu\big(A\setminus \{h<1\}\big)\ge (1-t)a,\\
\mu\{\indicatrice{A} h>0\}
& \le a\le\frac12.
\end{align*}
As $J$ is non-decreasing on $(0,\frac12)$, concave, and $J(0)=0$,
\[
J\left(\int \indicatrice{A} h \,d\mu \right)\ge J\big((1-t)a\big)\ge (1-t) J(a).
\]
Besides $J$ cancels at 0 and 1, and reaches its maximum at $\frac12$, so
\begin{align*}
\int J(\indicatrice{A} h) \,d\mu 
&\le J({\textstyle \frac12}) \ 
\mu\{0<\indicatrice{A}h<1\}\\
&\le J({\textstyle \frac12})\Big(\mu\{\indicatrice{A} h>0\}-\mu\{\indicatrice{A} h=1\}\Big)\\
&\le J({\textstyle \frac12})\, ta.
\end{align*}
 As for the third term of \eqref{ineq-fh}, it is bounded by
\begin{equation*}
\frac{\nu\big([r_1,r_2]\big)}{C_{\nu}(r_2-r_1)}\le \frac {ta}{K}.
\end{equation*}
For $\varepsilon>0$, we approximate $\indicatrice{A}$ by a smooth function $f_\varepsilon:\R^n\to[0,1]$
with  $f_\varepsilon=1$ on $A$ and  $f_\varepsilon=0$ outside $A_\varepsilon$. Then we apply \eqref{ineq-fh} to $f_\varepsilon$ and let $\varepsilon$ to 0,  taking advantage of the continuity of $J$:
\begin{equation*}
\sqrt{2}\max\Big(\frac1{C_{\nu}},\frac{\kappa r_2}{C_{\sigma_{n-1}}}\Big)
\mu^+(\partial A)\ge
\kappa^2(1-t) J(a)
- \left(J({\textstyle \frac12})+\frac1K\right)ta.
\end{equation*}
 Now by concavity,
 $J(a)\ge 2 J(\frac12) a$ on $\big[0,\frac12\big]$. Hence
\begin{align*}
\sqrt{2}\max\Big(\frac1{C_{\nu}},\frac{\kappa r_2}{C_{\sigma_{n-1}}}\Big)
\mu^+(\partial A)&\ge
\bigg(\kappa^2(1-t)- \frac{J(\frac12)+\frac1K}{2J(\frac12)}t \bigg) J(a)\\
&= \bigg(\kappa^2-t\Big(\kappa^2 +\frac12 +\frac1{2KJ(\frac12)}\Big) \bigg) J(a).
\end{align*}  
Taking for instance $K=\big(J(\frac12)\big)^{-1}$ and $t={\kappa^2}/({2(\kappa^2+1)})$ yields a non-trivial result.

Note that looking at closed sets was not a real restriction. Indeed, if $\liminf_{\varepsilon\to0^+} \mu(A_\varepsilon)-\mu(A) >0$ then $\mu^+(\partial A) =+\infty$.
\end{proof}

\section{Getting functional inequalities}\label{sec:func}

To apply Proposition \ref{thm:tensor} to our case, we need to know how to pass from an isoperimetric inequality to a functional inequality. 
 Actually we can approximate $L_\phi$ by an other profile satisfying the hypotheses made in Section \ref{sec:tensor},  assuming  furthermore that $\sqrt{\phi}$ is concave.

This new profile appears to be the isoperimetric function $\Is_{\mu_{1,\phi}}$  of $\mu_{1,\phi}$, denoted by $I_\phi$ henceforth. 
\begin{lemme}\label{lem:LI}
There exist universal constants $d_1>0$ and $d_2>0$ such that for all $\phi$ satisfying (H1'),
\[
d_1I_\phi\le L_\phi\le d_2 I_\phi.
\]
\end{lemme}
The second inequality is a consequence of Proposition 2.3 from \cite{milman-sodin} by Milman and Sodin, up to the uniform estimation of the normalizing constant of $\mu_{1\phi}$. However we give a self-contained proof of Lemma \ref{lem:LI} at the end of this section for completeness.
In the next lemma, $I_\phi$ is shown to satisfy the required properties.
\begin{lemme}\label{lem:Iphi-func}
Let $\phi$ satisfying (H1'). \begin{enumerate}[i)]
\item\label{propIphi} The function $I_\phi$ is continuous and concave on $[0,1]$, symmetric with respect to $1/2$, and $I_\phi(0)=I_\phi(1)=0$. 
\item\label{acceptable} Let $\mu$ be a measure on $\R^d$ and $C\ge0$. 
If 
\[
\Is_\mu\ge C I_\phi,
\]
then for all smooth functions $f:\R^d\to[0,1]$,
\[
\kappa I_\phi\left(\int f\,d\mu\right)\leq 
\int I_\phi(f)\,d\mu + \frac1{C} \int |\nabla f|\,d\mu,
\]
where $\kappa>0$ is a universal constant.
\end{enumerate}
\end{lemme}

\begin{proof}
Let us first remark that $\mu_{1,\phi}$ is an even log-concave probability measure on the real line. Hence half-lines solve the isoperimetric problem and we can express explicitly $I_\phi$ (see {e.g.} \cite{bobkov-extr-logconc}). Let $f_\phi:x\mapsto \frac{e^{-\phi(|x|)}}{Z_\phi}$ be the density of $\mu_{1,\phi}$, $F_\phi(x)=\mu_{1,\phi}\big\{(-\infty,x)\big\}$ its cumulative distribution function,  and  $G_\phi(x)=\mu_{1,\phi}\big\{(x,+\infty)\big\}$. Then 
\[
I_\phi=f_\phi\circ F_\phi^{-1}= f_\phi\circ G_\phi^{-1}
\]
and the properties stated in \ref{propIphi}) are clearly satisfied.
Besides the transfer principle emphasized by Barthe in \cite{barthe-level} holds : if $\Is_\mu\ge cI_\phi$ then $\mu$ satisfies  essentially the same  functional inequalities as $\mu_{1,\phi}$. 
 As a consequence, it remains to establish that for all smooth functions $f:\R\to[0,1]$,
\[
\kappa I_\phi\left(\int f\,d\mu_{1,\phi}\right)\leq 
\int I_\phi(f)\,d\mu_{1,\phi} +  \int | f'|\,d\mu_{1,\phi}.
\]
Now, applying the 2-dimensional isoperimetric inequality to the set
\[
\left\{(x,y)\in\R^2; y\le F_\phi^{-1}(f(x))\right\},
\]
one can show (see e.g. \cite{Barthe-spher}) that
\[
\Is_{{\mu_{1,\phi}}^{\otimes2}}\left(\int f\,d\mu_{1,\phi}\right)\leq 
\int I_\phi(f)\,d\mu_{1,\phi} +  \int | f'|\,d\mu_{1,\phi}.
\]
So, \ref{acceptable}) is shown if there exists a  universal $\kappa>0$  such that
\[
\Is_{{\mu_{1,\phi}}^{\otimes2}}\ge \kappa I_\phi.
\]
Actually,  a stronger dimension-free inequality holds and is stated in the next lemma.
\begin{lemme} \label{lem:tensorIphi} There exists $\kappa>0$ such that for all $\phi$ satisfying (H1') and all $n$,
\[
\Is_{{\mu_{1,\phi}}^{\otimes n}}\ge \kappa I_\phi.
\]
\end{lemme}
Barthe, Roberto, and Cattiaux prove it in \cite{BCR-isop} without verifying the universality of $\kappa$. However one can check that their constant can be uniformly controlled  for every $\phi$ satisfying (H1'), by using the same estimates for $G_\phi$ and $Z_\phi$ as in the proof of Proposition Lemma \ref{lem:LI}. Indeed,  a Beckner inequality  is shown to hold with a constant uniform in $\phi$, thanks to their explicit bound. It tensorizes and implies a super-Poincaré with a constant uniform in $n$ and $\phi$, which translates into the isoperimetric inequality of Lemma \ref{lem:tensorIphi}.

One can also check the simple criterion given by E. Milman in \cite{EMilmanConvFunctIsopIneq} for a tensorization result. As the function defined by
\[
t\mapsto \frac{L_\phi(t)}{L_2(t)}=\frac{\sqrt{\log\frac1t}}{\phi^{-1}\left(\log\frac1t\right)}
\]
is non-decreasing under (H1) --- all the more under (H1') --- then by Lemma \ref{lem:LI} there exists a universal constant $D>0$ such that
\[\forall 0<t\le s\le\frac12,\quad \frac{I_\phi(t)}{L_2(t)}\le D\frac{I_\phi(s)}{L_2(s)}.
\]
This also implies Lemma \ref{lem:tensorIphi}. So, up to the proof of Lemma \ref{lem:LI}, we are done.
\end{proof}

\begin{proof}[Proof of Lemma \ref{lem:LI}]
We can restrict ourselves to the case  $\phi(1)=1$. Indeed if we set $\phi_\lambda(x)=\phi(\lambda x)$, one can show $L_{\phi_\lambda}=\lambda L_\phi$ and $I_{\phi_\lambda}=\lambda I_\phi$.
This hypothesis  ensures that $1\le\phi'(1)\le2$ and also that
\[
t^2\le\phi(t)\le t \text{ on }[0,1] \quad\text{ and }\quad t\le\phi(t)\le t^2 \text{ on }[1,+\infty).
\]
Let $r\ge0$. By integration by part,
\[
\int_r^{+\infty} e^{-\phi}=\frac{e^{-\phi(r)}}{\phi'(r)} - \int_r^{+\infty} \frac{\phi''}{(\phi')^2}e^{-\phi}.
\]
By the properties of $\phi$ and especially as $(\sqrt{\phi})''\le0$,
\[
0\le \int_r^{+\infty} \frac{\phi''}{(\phi')^2}e^{-\phi} \le
\int_r^{+\infty} \frac{e^{-\phi}}{2\phi}
=
\frac{e^{-\phi(r)}}{2\phi(r)\phi'(r)} -
\int_r^{+\infty} \frac{(\phi')^2+\phi\phi''}{2(\phi\phi')^2}e^{-\phi}
\le
\frac{e^{-\phi(r)}}{2\phi(r)\phi'(r)}.
\]
Hence
\[
\frac{e^{-\phi(r)}}{\phi'(r)}\left(1-\frac1{2\phi(r)}\right)
\le
\int_r^{+\infty} e^{-\phi}
\le
\frac{e^{-\phi(r)}}{\phi'(r)}.
\]
In particular, if $r\ge1$,
\begin{equation}\label{encadr}
\frac{e^{-\phi(r)}}{2\phi'(r)}
\le
\int_r^{+\infty} e^{-\phi}
\le
\frac{e^{-\phi(r)}}{\phi'(r)}.
\end{equation}

Now let us estimate the normalizing constant for  $\mu_{1,\phi}$, denoted by $Z_\phi$.
\[
Z_\phi=2\int_0^{+\infty} e^{-\phi}=2\left(\int_0^{1} e^{-\phi}+\int_1^{+\infty} e^{-\phi}\right)\le 2(1 +e^{-1}).
\]
Moreover
\[
\int_0^{1} e^{-\phi}\ge \int_0^{1} e^{-x}\,dx = 1-e^{-1},
\]
so
\[
Z_\phi\ge 2(1- {e^{-1}})>1.
\]

By symmetry, we consider the case $a\in\left[0,\frac12\right]$. We set $a=G_\phi(r)=\int_r^{+\infty} \frac{e^{-\phi}}{Z_\phi}$. It follows that $r\ge0$. Then, to prove the lemma,  we need only to compare \[\frac{e^{-\phi(r)}}{Z_\phi}
\qquad \text{ with } \qquad
L_\phi\Big(G_\phi(r)\Big)=
 G_\phi(r)\frac{\log\frac1{G_\phi(r)}}{\phi^{-1}\left(\log\frac1{G_\phi(r)}\right)}.\]
 Recall that  $\phi(t)\le t\phi'(t) \le 2\phi(t)$ under (H1) so that 
\[
\frac12\phi'\circ\phi^{-1}(x)\le \frac x{\phi^{-1}(x)}\le \phi'\circ\phi^{-1}(x)
\]
and
\[
\frac{G_\phi(r)}2\, \phi'\circ \phi^{-1}\left(\log\frac1{G_\phi(r)}\right) 
\le
L_\phi\Big(G_\phi(r)\Big)
 \le
 {G_\phi(r)}\, \phi'\circ \phi^{-1}\left(\log\frac1{G_\phi(r)}\right).
\]
Assume first that $r\ge1$ so that \eqref{encadr} holds. On one hand,
\begin{align*}
L_\phi\Big(G_\phi(r)\Big)
 &\ge \frac{G_\phi(r)}2\, \phi'\circ \phi^{-1}\left(\log\frac1{G_\phi(r)}\right)\\
&\ge
\frac{e^{-\phi(r)}}{4Z_\phi\phi'(r)}  \phi'\circ \phi^{-1}\Big(\log\big( Z_\phi \phi'(r)e^{\phi(r)}\big)\Big)\\
&\ge
\frac{e^{-\phi(r)}}{4Z_\phi}\quad\text{since }Z_\phi \phi'(r)\ge 1 \text{ and } \phi'\circ \phi^{-1} \text{ is non-decreasing}.
\end{align*}
 On the other hand,
\begin{align*}
L_\phi\Big(G_\phi(r)\Big)
&\le G_\phi(r)\, \phi'\circ \phi^{-1}\left(\log\frac1{G_\phi(r)}\right)\\
&\le
\frac{e^{-\phi(r)}}{Z_\phi\phi'(r)}  \phi'\circ \phi^{-1}\Big(\log\big( 2Z_\phi \phi'(r)e^{\phi(r)}\big)\Big).
\end{align*}
One can show that
\begin{align*}
 2Z_\phi \phi'(r)e^{\phi(r)}
&\le 2Z_\phi r\phi'(r)e^{\phi(r)}
 \le 4Z_\phi \phi(r)e^{\phi(r)}
 \le \phi(4Z_\phi r)e^{\phi(r)}\\
&\le e^{\phi(4Z_\phi r)}e^{\phi(r)} 
 \le e^{2\phi(4Z_\phi r)}
 \le e^{\phi(8Z_\phi r)}.
\end{align*}
Thus
\begin{align*}
\frac{G_\phi(r)\log\frac1{G_\phi(r)}}{\phi^{-1}\left(\log\frac1{G_\phi(r)}\right)}
&\le
\frac{e^{-\phi(r)}}{Z_\phi\phi'(r)}  \phi'(8Z_\phi r)
\le32(1+e^{-1}) \frac{e^{-\phi(r)}}{Z_\phi}
.
\end{align*}
Now assume that $r\le1$. Let us remark that $L_\phi:a\mapsto \frac{a\log\frac1a}{\phi^{-1}\left(\log\frac1a\right)}$ is non-decreasing on $[0,\frac12]$. Indeed, ${L_\phi}'$ is of the same sign as
\begin{align*}
(x-1)\phi^{-1}(x)\phi'\circ\phi^{-1}(x)+x,\qquad \text{with }x=\log\frac1a.
\end{align*}
If $x>1$,
\[
(x-1)\phi^{-1}(x)\phi'\circ\phi^{-1}(x)+x\ge (x-1)x+x \ge0.
\]
Else $x\in[\log 2, 1]$ and
\begin{align*}
(x-1)\phi^{-1}(x)\phi'\circ\phi^{-1}(x)+x
\ge (x-1)2x+x=x(2x-1)\ge 0.
\end{align*}
So
\begin{align*}
L_\phi\Big(G_\phi(r)\Big)\ge L_\phi\Big(G_\phi(1)\Big) &\ge
\frac{G_\phi(1)}{2}\, \phi'\circ \phi^{-1}\left(\log\frac1{G_\phi(1)}\right)\\
&\ge\frac{e^{-\phi(1)}}{4Z_\phi\phi'(1)}  \phi'\circ \phi^{-1}\Big(\log\big( Z_\phi \phi'(1)e^{\phi(1)}\big)\Big)\\
&\ge
\frac{e^{-1}}{4Z_\phi}\ge
\frac{e^{-1}}4\ \frac{e^{-\phi(r)}}{Z_\phi}.
\end{align*}
Similarly for the lower bound,
\begin{align*}
L_\phi\Big(G_\phi(r)\Big)\le L_\phi\Big(G_\phi(0)\Big)&\le
G_\phi(0)\, \phi'\circ \phi^{-1}\left(\log\frac1{G_\phi(0)}\right)\\
&\le\frac12\phi'\circ \phi^{-1}(\log2)\le 1 
\le \frac{ 2(1 +e^{-1})}{e^{-1}}\ \frac{e^{-\phi(r)}}{Z_\phi}.
\end{align*}
\end{proof}

\section{Isoperimetry for $\mu_{n,\phi}$}\label{sec:isop-mu}
Now we can apply Proposition \ref{thm:tensor} to $\mu_{n,\phi}$ with $J=I_\phi$ when $\phi$ satisfies (H1') or $J=\Is_\gamma$ the Gaussian isoperimetric function when $\phi$ satisfies (H2). Indeed by Theorem \ref{thm:nu} and Lemma \ref{lem:LI}, 
\[\Is_{\nu_{n,\phi}}\ge C_{\nu_{n,\phi}} J\]
with $C_{\nu_{n,\phi}}=C \phi^{-1}(1)\frac{\sqrt{n}}{\phi^{-1}(n)}$ under (H1) and  $C_{\nu_{n,\phi}}=C \frac{\sqrt{n}}{\phi^{-1}(n)}$ under (H2), where $C>0$ is a universal constant. As for the sphere, it is known that $\sigma_{n-1}$ satisfies Gaussian isoperimetry with a constant of order $\sqrt{n}$, e.g. by a curvature-dimension criterion (cf \cite{livrelogsob}). That means that for every $a\le\frac12$ and every $\phi$ satisfying (H1),
\begin{align*}
\Is_{\sigma_{n-1}}(a)&\ge C \sqrt{n}\ \Is_\gamma(a)\\
& \ge CK\sqrt{n}\ a \sqrt{\log\frac{1}{a}}\\
&\ge CK\sqrt{\log 2}\sqrt{n}\ \phi^{-1}(1) \frac{\log\frac{1}{a}}{\phi^{-1}\left(\log\frac{1}{a}\right)}
\end{align*}
where $K>0$ is a universal constant.
\begin{proposition}\label{mu-big}
For every $c>0$, there exists $C>0$  such that if $e^{-c n}<\frac12$, then for every function $\phi$,
\begin{enumerate}[i)]
\item \label{mu-big-H1} if $\phi$ satisfies (H1') then
\[
\forall a \in \left[e^{-c n}, \frac12\right], \quad
 \Is_{\mu_{n,\phi}}(a)\ge C\frac{\sqrt{n}}{\phi^{-1}(n)}\  \phi^{-1}(1)  \frac{a\log\frac{1}{a}}{\phi^{-1}\left(\log\frac{1}{a}\right)}.
\]
\item if $\phi$ satisfies (H2) then
\[
\forall a \in \left[e^{-c n}, \frac12\right], \quad
\Is_{\mu_{n,\phi}}(a)\ge C \frac{\sqrt{n}}{\phi^{-1}(n)}\  a\sqrt{\log\frac1a}.
\]
\end{enumerate}
\end{proposition}
\begin{proof} We only prove \ref{mu-big-H1}). 
We can restrict ourselves to the case  $\phi(1)=1$. Let $\kappa$ be the constant coming from Lemma \ref{lem:Iphi-func}, then let $\kappa_1$ and  $\kappa_2$ be the corresponding constants given by Proposition \ref{thm:tensor}.
Set $c_1$ large enough to ensure
\[
c_1\ge2,
\]
\[
\max(\kappa_1,1) ec_1 e^{-c_1} \le e^{-c}.
\]
If we take $r_1=\phi^{-1}(c_1n)$, then we know by Lemma \ref{lem:queue-phi} that
\[
\kappa_1\nu_{n,\phi}\{(r_1,+\infty)\}\leq \kappa_1\left[ ec_1 e^{-c_1}\right]^n\le e^{-cn}.
\]
Here we use that $\phi^{-1}(c_1n)\le c_1 \phi^{-1}(n)$.
So for all $\phi$, for all $n$, and all $a\in\left[e^{-c n}, \frac12\right]$, Condition \eqref{cond-a} holds, i.e. 
\[\kappa_1\,\nu_{n,\phi}\{[r_1 ,+\infty)\}\le a \le \frac12.\] 
Now there exists a universal $C>0$ such that $C_{\nu_{n,\phi}} \ge C \frac{\sqrt{n}}{\phi^{-1}(n)}$ by Theorem \ref{thm:nu} and Lemma \ref{lem:Iphi-func} as explained at the beginning of the section (recall that here $\phi^{-1}(1)=1$).  So, if we set $r_2=(1+\frac{1}{CI_\phi(\frac12)})\phi^{-1}(c_1n)$, then Condition \eqref{cond-r} is also satisfied, i.e.
\[
  r_2-r_1 \ge \frac1{C_{\nu_{n,\phi}}I_\phi({\textstyle \frac12})}.
\]
Thus Proposition \ref{thm:tensor} yields
\[
\Is_{\mu_{n,\phi}}(a)\ge \kappa_2 \min \Big({C_{\nu_{n,\phi}}},\frac{C_{\sigma_{n-1}}}{r_2}\Big)\ 
I_\phi(a).
\]
Besides, there exists a universal $d>0$ such that $I_\phi\ge d\, L_\phi$ according to Lemma \ref{lem:LI}. In particular,
\[
I_\phi\left(\frac12\right)\ge d\, L_\phi\left(\frac 12\right)=  \frac{d\log2}{2\phi^{-1}(\log2)}\ge\frac{d\sqrt{\log 2}}2.
\]
We can deduce an upper bound for $r_2$. Finally, we have established
\[
\Is_{\mu_{n,\phi}}(a)\ge \kappa_2 Cd
\min \left(1,\left[c_1\left(1+\frac2{Cd\sqrt{\log2}}\right)\right]^{-1}\right)
\frac{\sqrt{n}}{\phi^{-1}(n)}\ 
L_\phi(a).
\]
\end{proof}

Therefore we have proved Theorem \ref{thm:muphi} at least for $a$ large enough. We complete the proof with Proposition \ref{prop:nu-small-h0} or Proposition \ref{prop:nu-small-h2} for smaller sets.

\section{Optimality and the isotropic case}\label{sec:opt}

One can ask whether the isoperimetric inequalities obtained are optimal at least up to universal constants, and whether we recover dimension-free results in the case of isotropic measures. 

We consider only bounds for the isoperimetric profile constructed as product of a function of $n$ times a function of $a$. When $\phi$ satisfies (H1'),  inequalities of Theorem \ref{thm:muphi} are optimal in $a$ for $n=1$, according to Lemma \ref{lem:LI}. In the supergaussian case, the central limit theorem for convex bodies of Klartag (see \cite{klartag-TCL}), in the simpler case of spherically symmetric distributions,  ensures that we cannot find a profile  bounding from below $\Is_{\mu_{n,\phi}}$ for all $n$, better than the Gaussian one (times a constant depending possibly on $n$). Else we should have concentration properties stronger than Gaussian. However by Klartag's theorem, there exists a sequence of positive number $\varepsilon_n\to0$ such that  for every Borel set $A\subset\R$ and every $r>0$,
\[
1-\mu_{n,\phi}\left(\left(A\times\R^{n-1}\right)_r\right)\ge
1-\gamma(A_r) - \varepsilon_n,
\]
 where $\gamma$ denotes the standard normal distribution. Thus we cannot have a rate of concentration valid for all $n$ better than the Gaussian one.

So optimal inequalities should be   of the  type
\begin{align*}
\forall a\in\left[0,\frac12\right],\quad
\Is_{\mu_{n,\phi}}(a) 
&\ge C_{\mu_{n,\phi}} (n)\ \phi^{-1}(1) \frac{a\log\frac{1}{a}}{\phi^{-1}\left(\log\frac{1}{a}\right)} 
& \text{ under (H1'),}\\
&\ge C_{\mu_{n,\phi}} (n)\ a\sqrt{\log\frac{1}{a}} 
& \text{ under (H2).}
\end{align*}
This implies
\[
\forall a\in\left[0,\frac12\right],\quad
\Is_{\mu_{n,\phi}}(a) 
\ge c\ C_{\mu_{n,\phi}} (n)\ a,
\]
where $c>0$ is universal.
Now Poincaré inequalities are equivalent up to universal constants to Cheeger inequalities (see \cite{EMilmanConvIsopSpectConc}), so the optimal constant in $n$ should be \[C_{\mu_{n,\phi}} (n)=C\sqrt{\frac{{n}}{\esp_{\mu_{n,\phi}}\left(|X|^2\right)}},\] in view of
 Theorem \ref{thm:bobkov-poinc}, with $C>0$  a universal constant.

Thus,  the two questions raised at the beginning of the section appear to be connected to the same property, namely $\esp_{\mu_{n,\phi}}\left(|X|^2\right)\simeq \left(\phi^{-1}(n)\right)^2$. Undoubtedly, this must be quite standard, nevertheless we state and prove the next lemma for completeness.
\begin{lemme}
\begin{enumerate}[i)]
\item 
Let $\phi$ be a function satisfying (H0). Define $r_n(\phi)$ the point where the density of the radial measure $\nu_{n,\phi}$ reaches its maximum, and $\esp_{\mu_{n,\phi}}|X|^2$ the second moment of $\mu_{n\phi}$.
For every $M>1$, there exists $n_0\in\N$ not depending on $\phi$ such that, for all $n\ge n_0$,
\[
  \frac{1}{M}\sqrt{\esp_{\mu_{n,\phi}}|X|^2}\le r_n(\phi) \le M\sqrt{\esp_{\mu_{n,\phi}}|X|^2}.
\]
\item Besides, if there exists $\alpha\ge1$ such that $x\mapsto\phi(x)/x^\alpha$ is non-increasing, then
\[
\phi^{-1}(n)\ge r_n(\phi) \ge e^{-\frac1e}\phi^{-1}(n).
\]
\end{enumerate}
\end{lemme}
\begin{proof}
To prove the first point, we can assume that $\mu_{n,\phi}$ is isotropic, that is to say that $\esp_{\mu_{n,\phi}}\left(|X|^2\right)=n$. Let $X$ be a random variable with distribution $\mu_{n,\phi}$. In the following, we denote by $\proba$, $\esp$, and $\var$ the corresponding probability, esperance, and variance. Let $\delta\in(0,1)$. In view of Lemma \ref{lem:klartag}, there exist universal constants $c>0$ and $C>0$ such that 
\[
\proba\left\{\big|r_n(\phi)-|X|\big|\ge \delta r_n(\phi) \right\} \le C e^{-cn\delta^2}.
\]
On the other hand, Bobkov proved in \cite{bobkov-radial-logconcave} the following upper bound for the variance of $|X|$ to establish Theorem \ref{thm:bobkov-poinc}: 
\[
\var|X|\le\frac{\left(\esp |X|\right)^2}{n},
\]
which can also be reformulate
\[
n \esp |X|^2 \le (n+1)  \left(\esp|X|\right)^2.
\]
Then 
\begin{align*}
\esp \left(\sqrt{\esp|X|^2}-|X|\right)^2&=2\sqrt{\esp|X|^2}\left(\sqrt{\esp|X|^2}-\esp|X|\right)\\
&\le2\left(\sqrt{1+\frac1n}-1\right)\sqrt{\esp|X|^2} \esp|X|\le \frac{\esp|X|^2}n.
\end{align*}
So, by Chebychev's inequality it holds for all $t>0$:
\[
\proba\left\{\big|\sqrt{\esp|X|^2}-|X|\big|\ge t \sqrt{\esp|X|^2} \right\} \le \frac{1}{nt^2}.
\]
Fix $\delta\in(0,1)$, and choose $n$ large enough to ensure $C e^{-cn\delta^2}+ {1}/{n\delta^2}<1$. Then there exist $x>0$ such that $|r_n(\phi)-x|\le \delta r_n(\phi)$ and 
$|\sqrt{\esp|X|^2}-x|\le \delta\sqrt{\esp|X|^2}$. It follows 
\[
  \frac{1-\delta}{1+\delta}\sqrt{n}\le r_n(\phi) \le \frac{1+\delta}{1-\delta}\sqrt{n}.
\]

Now, $r_n(\phi)$ satisfies $r_n(\phi)\phi'\big(r_n(\phi)\big)=n-1$. Therefore, as already mentioned, (H0) ensures that $r_n(\phi)\le \phi^{-1}(n)$. Assume moreover the existence of $\alpha\ge1$ such that $x\mapsto\phi(x)/x^\alpha$ is non-increasing. Then 
\[r_n(\phi)\ge \phi^{-1}\left(\frac{n-1}{\alpha}\right)
\ge \phi^{-1}\left(\frac{n}{2\alpha}\right)
\ge \left(\frac1{2\alpha}\right)^{\frac1{2\alpha}}\phi^{-1}(n)
\ge e^{-\frac1e}\phi^{-1}(n).
\]
\end{proof}
Eventually, we can state the following theorem.
\begin{theoreme}\label{thm:opt-iso}
\begin{itemize}
	\item If $\phi$ satisfies $(H1')$ or if $\phi$ satisfies $(H2')$, then the inequality proved in Theorem \ref{thm:muphi} is optimal.
	\item For any $n\in\N$, let us choose $\lambda>0$ such that $\mu_{n,{\phi_\lambda}}$ is isotropic, when replacing $\phi$ by $\phi_\lambda:x\mapsto\phi(\lambda x)$. Then it holds a dimension-free isoperimetric inequality. More precisely, there exist a universal $C>0$ and a universal $n_0\in\N$ such that
\begin{enumerate}[i)]
	\item  if $\phi$ satisfies $(H1')$ then
	\[
\forall a \in \left[0, \frac12\right], \quad \forall n\ge n_0,\quad
 \Is_{\mu_{n,\phi_\lambda}}(a)\ge C\  \phi^{-1}(1)  \frac{a\log\frac{1}{a}}{\phi^{-1}\left(\log\frac{1}{a}\right)}\ ;
\]
	\item if $\phi$ satisfies $(H2')$, then
	\[
\forall a \in \left[0, \frac12\right], \quad \forall n\ge n_0,\quad
 \Is_{\mu_{n,\phi_\lambda}}(a)\ge C\  a\sqrt{\log\frac{1}{a}}\ .
\]
\end{enumerate}
\end{itemize} 
\end{theoreme}

\paragraph{Acknowledgments.} I would like to thank Franck Barthe for his support and fruitful discussions, and also Emanuel Milman for  helpful suggestions.

\end{document}